\documentclass[11pt]{amsart}
\usepackage{amsxtra}
\usepackage{amssymb}
\usepackage{mathtools}
\usepackage{enumerate}
\usepackage{blindtext}
\usepackage[inline]{enumitem}
\usepackage{nicefrac}

\usepackage[pagewise]{lineno}

\mathtoolsset{showonlyrefs} 
\usepackage[dvipsnames]{xcolor}
\usepackage{mathrsfs}
\usepackage{tikz}
\tikzset{
  treenode/.style = {align=center, inner sep=0pt, text centered,
    font=\sffamily},
  arn_r/.style = {treenode, circle, blue, draw=blue,fill=yellow, 
    text width=1em, very thick},
    dot/.style={circle,draw,inner sep=1.2,fill=black},
}
\usepackage{forest}

\usetikzlibrary{shadings}
\usepackage{color}
\usepackage{hyperref}
\hypersetup{
  colorlinks   = true, 
  urlcolor     = blue, 
  linkcolor    = blue, 
  citecolor   = red 
}
\xdefinecolor{darkgreen}{RGB}{175, 193, 36}

\usepackage{amsfonts}
\usepackage{enumerate}
\usepackage{url}
\usepackage{color}      
\usepackage{epsfig}
\usepackage{graphicx}
\usepackage{esint}
\usepackage{euscript}
\usepackage{wasysym}

\newcommand{\T}{\mathbb{T}_m}
\renewcommand{\S}{\mathcal{S}}

\newtheorem{teo}{Theorem}[section]
\newtheorem{lem}[teo]{Lemma}
\newtheorem{co}[teo]{Corollary}
\newtheorem{pro}[teo]{Proposition}

\theoremstyle{remark}

\newtheorem{ex}[teo]{Example}

\theoremstyle{definition}
\newtheorem{de}[teo]{Definition}

\def\R{{\mathbb R}}

\def\N{{\mathbb N}}

\title[Quasiconvex functions on regular trees]{Quasiconvex functions on regular trees
}
\author[L. M. Del Pezzo, N. Frevenza and J. D. Rossi]
{Leandro M. Del Pezzo, Nicol{\'a}s Frevenza and Julio D. Rossi}

\address{Leandro M. Del Pezzo and Julio D. Rossi
\hfill\break\indent
CONICET and Departamento  de Matem{\'a}tica, FCEyN,
Universidad de Buenos Aires,
\hfill\break\indent Pabellon I, Ciudad Universitaria (1428),
Buenos Aires, Argentina.}

\email{{\tt ldpezzo@dm.uba.ar, jrossi@dm.uba.ar}}

\address{Nicol{\'a}s Frevenza 
\hfill\break\indent
Instituto de Estad\'{i}stica, FCEA,´
Universidad de la Rep{\'u}blica,
\hfill\break\indent
Gonzalo Ram{\'i}rez 1926, Montevideo, Uruguay.}

\email{{\tt nicolas.frevenza@fcea.edu.uy}}

\begin{document}

\begin{abstract}  
	We introduce a definition of a quasiconvex function on an infinite directed regular tree that depends on what we understand by a segment on the tree. 
	Our definition is based on thinking on segments as subtrees with the root as the midpoint of the segment and extends a previous notion of quasiconvexity on a tree.
	A convex set in the tree is then a subset such that it 
	contains every midpoint of every segment with terminal nodes in the set. Then, a quasiconvex function is a 
	real map on the tree such that every level set is a convex set.
    For this concept of quasiconvex functions on a tree, we show that given a continuous boundary datum, there 
    exists a unique quasiconvex envelope on the tree, and we characterize the equation that this envelope 
    satisfies. It turns out that this equation is a mean value property that involves a median among values of 
    the function on successors of a given vertex.
    We also relate the quasiconvex envelope of a function defined inside the tree to the solution of an 
    obstacle problem for this characteristic equation.
\end{abstract}

\thanks{Data sharing not applicable to this article as no datasets were generated or analysed during the current study}
\maketitle

\section{Introduction}

	Our main goal in this paper is to study quasiconvex functions on a regular directed tree.
	Let us start this introduction recalling the well-known definitions 
	of convexity and quasiconvexity in the Euclidean space. 
	A function  $u\colon S\to \mathbb{R}$ defined on a convex subset 
	$S\subset \mathbb{R}^N$ is called convex if for all $x,y\in S$ and 
	any $\lambda \in [0,1],$ we have
	\[
		u(\lambda x+(1-\lambda )y)\leq \lambda u(x)+(1-\lambda )u(y).
	\]
	That is, the value of the function at a point in the segment that joins 
	$x$ and $y$ y is less than or equal to the convex combination between 
	the values at the extrema. 
	An alternative way of stating convexity is to say that $u$ is convex 
	on $S$ if the epigraph of $u$ on $S$ is a convex set on $\mathbb{R}^{N+1}$. 
	We refer to \cite{Vel} for a general reference on convex structures. 

	A notion weaker than convexity is quasiconvexity. A function 
	$u\colon S\to \mathbb{R} $ defined on a convex subset $S$ 
	of the Euclidean space is called quasiconvex 
	if for all $x,y\in S$ and any $\lambda \in [0,1],$ 
	we have
	\[
		u(\lambda x+(1-\lambda )y)\leq \max \left\{ u(x),u(y)\right\}.
	\]
	An alternative geometric way of defining a quasiconvex function 
	$u$ is to require that each sublevel set $S_{\alpha }(u)=\{x\in S\colon u(x)\leq \alpha \}$
	is a convex set. See \cite{10} and citations therein for an overview.

	One problem with convexity is that whether
	or not a function is convex depends on the numbers which the function assigns
	to its level sets, not just on the shape of these level sets. The problem with this is that a
	monotone transformation of a convex function need not be convex. That is,
	if $u$ is convex and $g:\mathbb{R} \mapsto \mathbb{R}$ is increasing then 
	$g\circ u$ may fail to be convex. For instance, $f(x)=x^2$ is convex and 
	$g(x)=\arctan(x)$ is increasing but $g\circ f(x)$ is not convex.  
	However, the weaker condition, quasiconvexity, maintains this quality under monotonic
	transformations. Moreover, every monotonic transformation of a convex function
	is quasiconvex (although it is not true that every quasiconvex function
	can be written as a monotonic transformation of a convex function). 

	Quasiconvex functions have applications in mathematical analysis, 
	optimization, game theory, and economics.
	In nonlinear optimization, quasiconvex programming studies iterative methods 
	that converge to a minimum (if 
	one exists) for quasiconvex functions. 
	Quasiconvex programming is a generalization of convex programming.
	See \cite{Pearce} for an application to queueing theory on industrial organization. 
	In microeconomics, quasiconcave ($-u$ with $u$ quasiconvex) utility functions imply that consumers have 
	convex preferences, that is, the diversification of goods is preferred to the concentration on one of these. 
	Quasiconvex functions are important also in game theory and general equilibrium theory; in particular, in 
	Sion's theorem that asserts when we can interchange an infimum with a supremum, see \cite{Ko,Sion}.
	
	There is also a Partial Differential Equations (PDEs) approach for quasiconvex functions, 
	see \cite{BGJ12a,BGJ12b,BGJ13}. In fact, a function $u$ in the Euclidean space is quasiconvex if and only if it is a viscosity subsolution to 
	\begin{equation}\label{ec-RN}
		L(u) (x) \coloneqq \min_{ \substack{v \colon |v|=1, \\
		\langle v, \nabla u(x) \rangle =0}} 
		\langle D^2 u(x) v, v \rangle =0.
	\end{equation}
	Moreover, the quasiconvex envelope of a boundary datum inside a domain 
	is a solution to \eqref{ec-RN} and the quasiconvex envelope of a given function $g$ 
	inside the domain (defined as the largest quasiconvex function
	that is below $g$ in the domain) is the solution to the obstacle problem 
	(from above) for the operator $L$.

	When one wants to expand the notion of convexity or quasiconvexity to an ambient space beyond
	the Euclidean setting the key is to introduce what is a segment
	in our space and, once this is done, to understand what a midpoint in
the segment is.
	For extensions of convexity for graphs and lattices we refer 
	to \cite{Bapat-et-al,Ca,DPFRconvex,FaJa1,FaJa2,Fa,Mu1,Mu2,Mu3,Pel} and references therein.
	
	Here we want to set the ambient space to be a regular tree with $m-$branching that we will denote by $\T$.	
	This refers to a graph with a unique root and such that every node $x$ is connected with 
	$m+1$ nodes, it has exactly $m$ successors (we denote by $\S(x)$ the set of successors) 
	and only one ancestor (except the root that has only $m$ successors), 
	see the precise definition in the next section.

	In \cite{Bapat-et-al} a concept of convexity and quasiconvexity was introduced for a non-directed tree. 
In a finite tree, the authors proved that certain operations with convex and quasiconvex functions preserve the convex or quasiconvex structure. 
They also showed that certain functions of importance in the case of finite trees are convex or quasiconvex.
The regular tree is a discrete metric space where some techniques with analogies to the PDEs approach have been developed.
Recently, in \cite{DPFRconvex} the notion of convexity on $\T$ was extended as follows: 
	fix $k\in \{1,\dots,m\}$ and let $\mathbb{T}_{k}^{x}$ denote the collection
	of finite subgraphs of $\T$ with a root at $x$ and $k-$branching (every node that is not a terminal node
	has exactly $k$ successors). 
	For $\mathbb{B}\in \mathbb{T}_{k}^{x}$ we denote by $\mathcal{E}(\mathbb{B})$ the set of terminal nodes
	of $\mathbb{B}$. Then, a function $u:\T \to \mathbb{R}$ is called $k-$ary
	convex or $k-$convex if
    for any $x\in\T$
    \begin{equation}\label{convex.II}
        u(x) \leq \sum_{y\in\mathcal{E}(\mathbb{B})}\dfrac1{k^{|y|-|x|}} u(y),
        \qquad \forall \mathbb{B}\in\mathbb{T}_{k}^{x}.
    \end{equation}
    In this notion of convexity, the subtree $\mathbb{B}$ has the role of a segment; the midpoint is the root of $\mathbb{B}$, and the $k-$convexity property just says that the value of the 
    function $u$ at the midpoint is less than or equal to a weighted average of the values of $u$ at the endpoints.
It should be noted that the meaning of segment depends on $k$ and admit more than two endpoints.
    For $k=2$, this definition recovers the convex notion of \cite{Bapat-et-al}.

	Here, based on the previously mentioned idea of a segment and a midpoint in the tree $\T$, 
	we introduce a definition of a $k-$quasiconvex function on $\T$. 
	A $k-$convex set in the tree $C\subset \T$ is a subset that contains every midpoint of every segment with 
	terminal nodes in the set, that is, $C\subset \T$ is $k-$convex 
	if for every $\mathbb{B}\in\mathbb{T}_{k}^{x}$ with $\mathcal{E}(\mathbb{B})\subset C$ 
	we have that $x \in C$. Then, the natural definition for $k-$quasiconvexity runs as follows: 
	a function on the tree $u$ is $k-$quasiconvex if every sublevel set 
	$S_{\alpha }(u)=\{x\colon u(x)\leq \alpha \}$ is a $k-$convex set in $\T$. 

	First, we prove a characterization of being $k-$quasiconvex in terms of an inequality involving only the 
	values of $u$ at the successors of $x$, that is, as a local property. 
	A function $u$ is $k-$quasiconvex on the tree if and only if for every vertex $x \in \T$ it holds that
	\begin{equation}
		\label{def-quasiconvex-equiv.intro.99}
		u(x) \leq 
			\min_{\substack{y_1,\dots,y_k \in \S(x) \\ y_i \neq y_j}}\,
			 \max_{i=1,\dots,k} \left\{u(y_i)\right\}.
	\end{equation}
	Notice that the right side of~\eqref{def-quasiconvex-equiv.intro.99} is the $k-$th smallest value among all the values of $u$ at the set of successors of $x$, $\S(x)$.
	This characterization shows that the definition of quasiconvexity of \cite{Bapat-et-al} is equivalent to the 2-quasiconvexity introduced here.

    For this notion of $k-$quasiconvexity on a tree we show that given a 
    boundary datum $f$ on the boundary of the tree, there exists a unique $k-$quasiconvex envelope in $\T$ 
    (this $k-$quasiconvex envelope is defined as the supremum of $k-$quasiconvex functions that are below $f$ 
    on the boundary of the tree) and we characterize the equation that this envelope 
    satisfies: the $k-$quasiconvex envelope $u_f^*$ is the largest solution to 
    \begin{equation}
		\label{def-quasiconvex-equiv.intro}
		u(x) = \min_{\substack{y_1,\dots,y_k \in \S(x) \\ y_i \neq y_j}}\,
		 \max_{i=1,\dots,k} \left\{u(y_i)\right\}
	\end{equation}
	that is below $f$ on $\partial \T$. 
	Notice that here we have saturated the inequality \eqref{def-quasiconvex-equiv.intro.99}.
	For a bounded boundary datum $f,$ we prove existence and uniqueness for solutions to the problem 
	\eqref{def-quasiconvex-equiv.intro} and in the case where $f$ is continuous we show that the 
	solution of \eqref{def-quasiconvex-equiv.intro} attains the datum with continuity.  
	We also establish an analogy between this equation \eqref{def-quasiconvex-equiv.intro} and the associated equation \eqref{ec-RN} for the quasiconvexity in the Euclidean setting in Section \ref{sect-PDEs}.

	It turns out that this equation \eqref{def-quasiconvex-equiv.intro} is a mean value property that involves 
	the $k-$th order statistic of the values of the function on the successors of a given vertex.
	In the particular case of the $m-$branching directed tree with $m$ odd and $k=\frac{m-1}{2}$, 
	the equation~\eqref{def-quasiconvex-equiv.intro} is given by the median operator, 
	that is, the $k-$quasiconvex envelope is the largest solution to
	\begin{equation} \label{median.intro}
			u (x)  = \operatorname{median}
			\left\{ u(y) \colon y\in  \S(x) \right\}
			\quad\text{ for } x\in\T.
		\end{equation}
		In the special cases $k=1$ or $k=m$, the equation \eqref{def-quasiconvex-equiv.intro} reduces to
		\begin{equation} \label{special.cases.intro}
		\begin{array}{l}
			\displaystyle u (x)  =  \min_{y\in \S(x) } \left\{u(y)\right\}
			\quad\text{ for } k=1, \\[10pt]
			\displaystyle u (x)  =  
		 \max_{y \in \S(x)} \left\{u(y)\right\}
			\quad\text{ for } k=m.
			\end{array}
		\end{equation}
	Here, we concentrate on the more interesting case $k\in \{2,...,m-1\}$.
    
    We also relate the $k-$quasiconvex envelope of a function $g:\T \to \mathbb{R}$ 
    defined inside the tree to the solution of an obstacle problem
    for this characteristic equation \eqref{def-quasiconvex-equiv.intro}.
    
    
    \medskip
    
    The paper is organized as follows: in the next section we describe precisely
    the ambient space to be the regular tree with $m-$branching, set the notations that we are going to use
    and state the results; while in Section \ref{sect-proof} we gather the proofs;
     in Section \ref{sect-examples} we include as an example some computations and remarks
    showing that the quasiconvex envelope is easy to compute when the boundary datum $f$ is monotone;
    finally, in Section \ref{sect-PDEs} we look at the equation for the $k-$quasiconvex envelope, 
    \eqref{def-quasiconvex-equiv.intro} 
    in the special case $k=2$
    and compare it
    with the equation for the Euclidean case \eqref{ec-RN}.

\section{Settings, notations and statements}	
	Given $m\in\mathbb{N}_{\ge2},$ a tree $\T$ with regular $m-$branching is an infinite directed 
	graph with vertex set defined by the empty set $\emptyset$, called the root, and all finite 
	sequences $(a_1,a_2,\dots,a_l)$ with $l\in\N,$ whose coordinates $a_i$ are chosen from $\{0,1,\dots,m-1\}.$     
    \begin{center}
        \pgfkeys{/pgf/inner sep=0.19em}
        \begin{forest}
            [$\emptyset$,
                [0
                    [0
                        [0 [,edge=dotted]]
                        [1 [,edge=dotted]]
                        [2 [,edge=dotted]]
                    ]
                    [1
                        [0 [,edge=dotted]]
                        [1 [,edge=dotted]]
                        [2 [,edge=dotted]]
                    ]
                    [2
                        [0 [,edge=dotted]]
                        [1 [,edge=dotted]]
                        [2 [,edge=dotted]]
                    ]
                ]
                [1
                    [0
                        [0 [,edge=dotted]]
                        [1 [,edge=dotted]]
                        [2 [,edge=dotted]]
                    ]
                    [1
                        [0 [,edge=dotted]]
                        [1 [,edge=dotted]]
                        [2 [,edge=dotted]]
                    ]
                    [2
                        [0 [,edge=dotted]]
                        [1 [,edge=dotted]]
                        [2 [,edge=dotted]]
                    ]
                ]
                [2
                    [0
                        [0 [,edge=dotted]]
                        [1 [,edge=dotted]]
                        [2 [,edge=dotted]]
                    ]
                    [1
                        [0 [,edge=dotted]]
                        [1 [,edge=dotted]]
                        [2 [,edge=dotted]]
                    ]
                    [2
                        [0 [,edge=dotted]]
                        [1 [,edge=dotted]]
                        [2 [,edge=dotted]]
                    ]
                ]    
            ]
        \end{forest}
        
       A regular tree with $3-$branching.
    \end{center}
	The edge structure is defined as follows: each vertex $x$ has $m$ successors, obtained by adding 
    another coordinate to $x$. 
    We denote by 
    \[
        \S(x)\coloneqq\{(x,i)\colon i\in\{0,1,\dots,m-1\}\}
    \]
    the set of successors of the vertex $x.$ 
    If $x$ is not the root, then $x$ has a only an 
    immediate predecessor, which is indicated by $\hat{x}.$
	A vertex $x\in\T$ has level $l\in\mathbb{N}$ if $x=(a_1,a_2,\dots,a_l)$.   
    The level of $x$ is denoted by $|x|.$

    A branch of $\T$ is an infinite sequence of vertices, where each one of them is 
    followed by one of its immediate successors.
    The collection of all branches defines the boundary of $\T$, denoted 
    by $\partial\T$.
    Note that the function $\psi:\partial\T\to[0,1]$ defined as
    \[
        \psi(\pi)\coloneqq\sum_{j=1}^{+\infty} \frac{a_j}{m^{j}}
    \]
    is surjective, where $\pi=(a_1,\dots, a_j,\dots)\in\partial\T$ and
    $a_j\in\{0,1,\dots,m-1\}$ for all $j\in\mathbb{N}.$ 
    Whenever $x=(a_1,\dots,a_j)\in\T$ is a vertex, we set
    \[
        \psi(x)\coloneqq\psi(a_1,\dots,a_j,0,\dots,0,\dots).
    \]
    Each vertex $x$ has associated an interval $I_x$ of length $\tfrac{1}{m^{|x|}}$ as follows
    \[
        I_x\coloneqq\left[\psi(x),\psi(x)+\frac1{m^{|x|}}\right].
    \]
    Observe that for all $x\in \T$, $I_x \cap \partial\T$ is the subset of $\partial\T$ 
    formed by all branches that pass through $x$.
    Additionally, for any branch $\pi=(a_1,\dots, a_j,\dots)\in\partial\T,$ 
    we can associate the sequence of intervals $\{I_{\pi,j}\}$ given by
    \[
        I_{\pi,j}\coloneqq I_{x_j}\quad \text{with }x_j=(a_1,\dots,a_j) \text{ for all } j.
    \]
    It is easy to see that $I_{\pi,j+1}\subset I_{\pi,j}$ and $\psi(\pi)\in I_{\pi,j}$ for all $j.$

\subsection{Quasiconvexity for directed regular trees} 
First, let us recall the definition of a $k-$convex set inside the tree. 
Fix $k\in \{2,\dots,m-1\}$.
We denote by $\mathbb{T}_{k}^{x}$ the collection of finite directed subgraphs of $\T$ with $x$ 
as root and $k-$branching and for
$\mathbb{B}\in \mathbb{T}_{k}^{x}$ we write $\mathcal{E}(\mathbb{B})$ for the set of terminal nodes of $\mathbb{B}$. 
		
		\begin{de} \label{def.convex.set}
	A set $C\subset \T$ is $k-$convex if for every $\mathbb{B}\in\mathbb{T}_{k}^{x}$ 
	with $\mathcal{E}(\mathbb{B})\subset C$ we have that $x \in C$. 
	\end{de}
	
	Then, the definition of a quasiconvex function runs as follows.
	\begin{de} \label{def.quasiconvex}
		A function on the tree $u:\T \mapsto \mathbb{R}$ is called $k-$quasiconvex 
		if every sublevel set $S_{\alpha }(u)=\{x\colon u(x)\leq \alpha \}$ is a $k-$convex set in $\T$.
	\end{de}

		We can characterize quasiconvexity by an inequality involving only the values of 
	$u$ at the node and its successors. 
	This characterization will be used to argue why the previous definitions of convexity and quasiconvexity seem ``natural'', establishing an analogy with the Euclidean case, which is described in Section \ref{sect-PDEs}.

	\begin{teo} \label{teo.charac.local}
		A function $u$ is $k-$quasiconvex if and only if 
		\begin{equation}\label{def-quasiconvex-equiv}
			u(x) \leq 
			\min_{\substack{y_1,\dots,y_k \in \S(x) \\ y_i \neq y_j}} \,
			\max_{i=1,\dots,k} \left\{u(y_i)\right\}	\qquad 
			\text{for every }x \in \T.
		\end{equation}
	\end{teo}

	Notice that the right side of~\eqref{def-quasiconvex-equiv} is the 
	$k-$th smallest value among all the values of $u$ in the successors of $x$.

	\subsection{The quasiconvex envelope of a boundary datum} 
		We are interested in the \emph{$k-$quasi\break~convex envelope} of a function defined on $\partial \T$.
		Given $f\colon[0,1]\to\mathbb{R},$ the $k-$quasiconvex envelope of $f$ on $\T$ is defined as follows
		\begin{equation} 
			\label{quasiconvex-envelope-arbol}
				u^*_f (x) \coloneqq 
				\sup \left\{u(x) \colon u\in\mathcal{QC}_k(f) \right\},
			\end{equation}
		where
		\begin{equation}
		    \mathcal{QC}_k(f)\coloneqq
		    \left\{ u\colon\T\to\mathbb{R}\colon 
				u \text{ is $k-$\emph{quasiconvex} and } 
				\limsup_{x\to \pi\in \partial \T} u(x)\leq f(\psi(\pi)) 
				\quad \forall \pi\in \partial \T
			\right\}.
	\end{equation}
	
	The $k-$quasiconvex envelope is unique by the fact that the maximum of two $k-$quasiconvex functions is also $k-$quasiconvex. 
	In the next theorem we characterize the $k-$quasiconvex envelope as the largest solution of the nonlinear 
	inequality \eqref{def-quasiconvex-equiv} on $\T$ that is below $f$ on $\partial \T$.	
		
	\begin{teo} \label{teo-quasiconvex} 
		Given a bounded function $f\colon[0,1]\to\mathbb{R}$, its $k-$quasiconvex envelope $u^*_f$ is unique 
		and is given by the largest solution to
		\begin{equation} 
		\label{eq-tree-quasiconvex}
			\begin{cases}
				\displaystyle u (x)  =
				\min_{\substack{y_1,\dots,y_k \in \S(x) \\ y_i \neq y_j}} \, 
				\max_{i=1,\dots,k} \left\{u(y_i)\right\}&\text{ for } x\in\T, \\
				\displaystyle
				u(\pi) \coloneqq \limsup_{x\to \pi} u(x) \leq f(\psi(\pi)) 
				& \text{ for } \pi\in \partial \T
			\end{cases}
		\end{equation}
		
		Moreover, the corresponding boundary value problem for the equation in \eqref{eq-tree-quasiconvex} on $\T$ 
		with a continuous Dirichlet datum $f$ on $\partial \T$ has existence and uniqueness, 
		that is, the $k-$quasiconvex envelope $u^*_f$ reaches $f$ on $\partial \T$ when $f$ is continuous
		in the sense that $ \lim_{x\to \pi} u(x) = f(\psi(\pi))$ for $\pi\in \partial \T$. 
	\end{teo}

	A natural question is to compare solutions to these solutions for different $k$'s but 
	the same boundary datum. The next comparison principle goes in this direction and the immediate 
	corollary provides an answer for the behavior of the solutions.

	\begin{teo} 
		\label{cp-general} 
			Fix $k,j\in \{2,\dots,m-1\}$ with $k\geq j$. 
			Let $u$ and $v$ satisfy
				\begin{equation}
					\label{cp-ec-interior}
						\displaystyle u(x) \geq 
						\min_{\substack{y_1,\dots,y_k \in \S(x) \\ y_i \neq y_l}}\, 
						\max_{i=1,\dots,k} \left\{u(y_i)\right\} 
						\quad \text{ and } \quad \displaystyle v(x) \leq 
						\min_{\substack{y_1,\dots,y_j \in \S(x) \\ y_i \neq y_l}}\, 
						\max_{i=1,\dots,j} \left\{v(y_i)\right\} 
				\end{equation}
			for every $x\in\T$, together with
		\begin{equation}
			\label{cp-ec-borde}
				\limsup_{x\to \pi} u(x) 
				\geq 
				\liminf_{x\to \pi} v(x) 
				\quad \forall \pi\in \partial \T.
		\end{equation}
		Then, 
		\[
		u(x)\geq v(x)
		\] for all $x\in\T$.
	\end{teo}	

	As an immediate corollary, we get that the $k-$quasiconvex envelopes for different values of $k$ 
	are ordered. 
	\begin{co} 
		\label{corol.compar}
			Fix $k,j\in \{2,\dots,m-1\}$ with $k\geq j$.
			Let $f,g\colon[0,1]\to \R$ be continuous functions with $f\geq g$, 
			$u$ and $v$ be the unique solutions of the equations
			\begin{equation}
				\label{corol.compar.equations}
					\displaystyle u(x) = 
					\min_{\substack{y_1,\dots,y_k \in \S(x) \\ y_i \neq y_l}} \max_{i=1,\dots,k} 
					\left\{u(y_i)\right\} 
					\quad 
					\text{ and } 
					\quad 
					\displaystyle v(x) = 
					\min_{\substack{y_1,\dots,y_j \in \S(x) \\ y_i \neq y_l}} \max_{i=1,\dots,j} 
					\left\{v(y_i)\right\},
			\end{equation}
			for every $x\in\T$, with $f$ and $g$ as boundary data, respectively.
			Then,
			\[
				u(x)\geq v(x)
			\]
			for all $x\in\T.$
		\end{co}

\subsection{The quasiconvex envelope of a function inside $\T$}
	We also study the $k-$quasiconvex envelope of a bounded function $g\colon \T \to \R,$ 
	that is, we consider
	\begin{equation}
		\label{solucion-obstaculo}
		u_g^{\star} (x) \coloneqq \sup \left\{u(x) \colon u\in \mathfrak{QC}_k(g)\right\},
	\end{equation}
	where
	\begin{equation}
		\mathfrak{QC}_k(g) \coloneqq \left\{u\colon \T\to \R \colon u 
		\text{ is quasiconvex and } u(x)\leq g(x) \quad \forall x\in\T \right\}.
	\end{equation}
	Observe that $g$ is not necessarily $k-$quasiconvex 
	(when $g$ is $k-$quasiconvex, then we trivially have $u_g^{\star}  \equiv g$).

	The quasiconvex envelope $u_g^{\star} $ is also unique.
	One can characterize $u_g^{\star} $ as the solution to the obstacle problem for the equation
	\eqref{def-quasiconvex-equiv}.
	This property is analogous to the convex envelope on the Euclidean space and the regular tree 
	(see for instance \cite{Ober33,DPFRconvex}).
	A relevant set for this type on envelopes is the coincident set, \emph{i.e.}, 
	the set where the $k-$quasiconvex envelope $u_g^{\star} $ hits the obstacle $g$,
	\[
		CS(g) \coloneqq \left\{x\in\T \colon u_g^{\star} (x) = g(x) \right\}.
	\]

	These aspects are summarized in the next result.

	\begin{teo}
		\label{teo-obstaculo}
			The $k-$quasiconvex envelope $u_g^{\star} $ of a function 
			$g\colon \T \to \R$ that is bounded below is the largest solution to the problem
			\begin{equation}
				\label{obstaculo-ecuacion}
				\begin{cases}
					\displaystyle u (x)  \leq 
					\min_{\substack{y_1,\dots,y_k \in \S(x) \\ y_i \neq y_l}} \max_{i=1,\dots,k}  
					\left\{u(y_i)\right\}  &\text{ for } x\in\T\\
					u(x) \leq g(x) & \text{ for } x \in \T.
				\end{cases}
			\end{equation}
		
		For vertices inside $CS(g)$ the obstacle $g$ verifies the inequality
		\begin{equation}
			\label{coincidence-set}
				\displaystyle g(x)\leq 
				\min_{\substack{y_1,\dots,y_k \in \S(x) \\ y_i \neq y_l}}\, \max_{i=1,\dots,k} 
				\{g(y_i)\},
		\end{equation}
		while outside $CS(g)$ the $k-$quasiconvex envelope $u_g^{\star} $ satisfies the equation 
		\begin{equation}
			\label{fuera-coincidence-set}
			u_g^{\star} (x) = 
			\displaystyle 
			\min_{\substack{y_1,\dots,y_k \in \S(x) \\ y_i \neq y_l}} 
			\,\max_{i=1,\dots,k}  \left\{u_g^{\star}  (y_i)\right\}.
		\end{equation}
	\end{teo}

	In this setting, we also have a comparison result, analogous to Corollary \ref{corol.compar}.

	\begin{co} \label{corol.compar.interior}
		Fix $k,j\in \{2,\dots,m-1\}$ with $k\geq j$.
		Given $g\colon \T \to \R$  a bounded function, 
		let $u$ and $v$ be the unique quasiconvex envelopes 
		for $k$ and $j$ respectively, with $g$ as interior datum for both cases.
		Then, 
		\[
			u(x)\geq v(x)
		\] 
		for all $x\in\T.$
	\end{co}

\section{Proofs} \label{sect-proof}
	Let us start by proving the following proposition.
	\begin{pro} \label{equiv.1} 
		A function $u\colon \T \mapsto \mathbb{R}$ is $k-$quasiconvex 
		if and only if it holds that
 		\begin{equation}\label{quasi.convex.II}
        	u(x) \leq \max_{y\in\mathcal{E}(\mathbb{B})} u(y),
  		\end{equation}
		for every $x\in \T$ and every $\mathbb{B}\in\mathbb{T}_{k}^{x}.$ 
			\end{pro}

	\begin{proof} 
		First, assume that $u$ is $k-$quasiconvex and take any $x\in \T$ and any 
		$\mathbb{B}\in\mathbb{T}_{k}^{x}$. 
		Then, consider the sublevel set $$S_{\alpha }(u)=\{y\colon u(y)\leq \alpha \}$$
		with $$\alpha = \max\limits_{y\in\mathcal{E}(\mathbb{B})} u(y).$$ 
		This set $S_{\alpha }(u)$ is $k-$convex in $\T$ 
		since $u$ is $k-$quasiconvex, and so, 
		every terminal node in $\mathbb{B}$ belongs to $S_{\alpha }(u)$.
		Hence, we get that $x \in S_{\alpha }(u)$, that is, 
		\[
			u(x) \leq \alpha = \max_{y\in\mathcal{E}(\mathbb{B})} u(y).
		\]
	
		To see the converse, let $u$ be a function such that \eqref{quasi.convex.II} 
		holds and consider a sublevel set 
		$S_{\alpha }(u)=\{y\colon u(y)\leq \alpha \}$. 
		Let $\mathbb{B}$ be a finite subtree with $k-$branching with terminal nodes that belonging to 
		$S_{\alpha }(u)$, that is $\mathcal{E}(\mathbb{B})\subset S_{\alpha }(u)$. 
		Then, from \eqref{quasi.convex.II}, if we denote by $x$ the root of $\mathbb{B}$ we have
		\[
			u(x) \leq  \max_{y\in\mathcal{E}(\mathbb{B})} u(y) \leq \alpha.
		\]
		This shows that $S_{\alpha }(u)$ is a $k-$convex set 
		and proves that $u$ is $k-$quasiconvex since its sublevel sets are convex. 
	\end{proof}

	We can characterize $k-$quasiconvexity by an inequality that involves only the values of $u$ at the 
	successors of the point $x$ (a local property). 

	\begin{pro}  \label{equiv.2}
		A function $u \colon \T\to\R$ is $k-$\emph{quasiconvex} if and only if for every $x \in \T$ 
		and for any $k$ different successors of $x$, $y_1,\dots,y_k$ with $y_i \in \S(x)$, it holds
		\begin{equation}
			\label{def-quasiconvex}
				u(x) \leq \max_{i=1,\dots,k} \left\{u(y_i) \right\}.
		\end{equation}
	\end{pro}

	\begin{proof}
		Assume that $u$ is quasiconvex and choose the subtree $\mathbb{B}$ composed by $x$ 
		as root and any set of $k$ different successors of $x$,  $y_1,...,y_k$, 
		$y_i \in \S(x)$. Then, using Proposition \ref{equiv.1}, we get that 
		\[
 			u(x) \leq \max_{y\in\mathcal{E}(\mathbb{B})} u(y) = \max_{i=1,...,k} \left\{u(y_i)\right\},
		\]
		since the terminal nodes of $\mathbb{B}$ are $y_1,...,y_k$.

		To prove the converse, take any subtree $\mathbb{B}\in\mathbb{T}_{k}^{x}$ and iterate  
		the inequality 
		\[
				u(x) \leq \max_{i=1,...,k} \left\{u(y_i)\right\}
		\]
		to obtain 
		\[
			u(x) \leq \max_{y\in\mathcal{E}(\mathbb{B})} u(y).
		\]
		Then, using again Proposition \ref{equiv.1}, we conclude that $u$ is $k-$quasiconvex.
	\end{proof}

		From our previous result Theorem \ref{teo.charac.local} follows immediately. 

	\begin{proof}[Proof of Theorem \ref{teo.charac.local}]
		The inequality \eqref{def-quasiconvex} is equivalent to 
		\begin{equation}
			\label{def-quasiconvex-equiv.22}
			u(x) \leq 
			\min_{\substack{y_1,\dots,y_k \in \S(x) \\ y_i \neq y_j}}\, \max_{i=1,\dots,k} 
			\left\{u(y_i)\right\}.
		\end{equation}
		Therefore, \eqref{def-quasiconvex-equiv.22} characterizes $k-$quasiconvexity.
	\end{proof}

\subsection{The quasiconvex envelope of a boundary datum}
	We start proving the comparison principle stated in Theorem~\ref{cp-general}.

	\begin{lem} 
		Fix $k,j\in \{2,\dots,m-1\}$ with $k\geq j$. 
		Let $u$ and $v$ satisfy
		\begin{equation}
			\label{lema-cp-ec-interior}
				u(x) \geq 
				\min_{\substack{y_1,\dots,y_k \in \S(x) \\ y_i \neq y_l}} \max_{i=1,\dots,k} 
				\left\{u(y_i)\right\} 
				\quad \text{ and } \quad  
				v(x) \leq \min_{\substack{y_1,\dots,y_j \in \S(x) \\ y_i \neq y_l}} \max_{i=1,\dots,j} 
				\left\{v(y_i)\right\} 
		\end{equation}
		for all $x\in\T$ and
		\begin{equation}
			\label{lema-cp-ec-borde}
				\limsup_{x\to \pi} u(x) 
				\geq 
				\liminf_{x\to \pi} v(x) 
				\quad \forall \pi\in \partial \T.
		\end{equation}
		Then, $u(x)\geq v(x)$ for all $x\in\T$.
	\end{lem}	

	\begin{proof}
		Adding $c>0$ to $u$ we may assume that~\eqref{lema-cp-ec-borde} is strict. Assume that
		\[
			M = \sup_{x\in\T} \left(v(x) - u(x)\right) > 0.
		\]
		By the strict version of \eqref{lema-cp-ec-borde} the supremum 
		is a maximum and it is attained inside $\T$.  Let $x_0$ be a vertex with maximal level 
		$|x_0|$ and such that $v(x_0) - u(x_0) = M$.

		Now, by the equation \eqref{lema-cp-ec-interior} for $v$ we have that $v(x_0)$ is smaller 
		or equal to $v$ evaluated in $j$ nodes of $\S(x_0)$, so, there exists at least $m-j+1$ 
		successors of $x_0$ such that 
		\[
			v(y) \geq v(x_0).
		\]
		On the other hand, using \eqref{lema-cp-ec-interior} for $u$, 
		we deduce that for at least $k$ successors of $x_0$ it holds that
		\[ 
			u(y) \leq u(x_0).
		\]
		Since $k\geq j$, by the pigeonhole principle, there is some 
		$y\in\S(x_0)$ where both inequalities are satisfied.
		Then, at this particular $y\in\S(x_0)$ we have
		\[
			v(y) - u(y) \geq v(x_0) - u(x_0) = M,
		\]
		but this contradicts the assumption of maximal level for $x_0$. 
		So, we conclude that 
		\[
			u(x)\geq v(x)
		\] 
		as we wanted to show.
	\end{proof}

	The proof of Theorem \ref{teo-quasiconvex} is split into two lemmas 
	and an easy application of the comparison principle. In the first lemma, we show that $u_f^*$ 
	is well-defined, $k-$quasiconvex and reaches $f$ on $\partial \T$ when $f$ is continuous.
	Then, we prove that $u_f^*$ is the largest solution of~\eqref{eq-tree-quasiconvex}. 
	The uniqueness of the solution for the equation~\eqref{eq-tree-quasiconvex} is deduced from the comparison 
	principle Theorem \ref{cp-general}. 

	\begin{lem} 
		Let $f\colon[0,1]\to\mathbb{R}$ be a bounded function
	 	and let $u^*_f$ be given by
		\begin{equation} 
			u^*_f (x) \coloneqq 
			\sup \left\{u(x) \colon u\in\mathcal{QC}_k(f) \right\},
		\end{equation}
		where
		\begin{equation}
	    \mathcal{QC}_k(f)\coloneqq
	    \left\{ u\colon\T\to\mathbb{R}\colon 
			u \text{ is $k-$\emph{quasiconvex} and } 
			\limsup_{x\to \pi\in \partial \T} u(x)\leq f(\psi(\pi)) 
		\right\}.
		\end{equation}
		Then, $u_f^*$ is well-defined, unique, $k-$quasiconvex 
		and it is below $f$ at the boundary, i.e.,
		\begin{equation}
			\label{pepe}
			u_f^*(\pi) \coloneqq \limsup_{x\to \pi} u_f^*(x) \leq f(\psi(\pi))
		\end{equation}
		for every $\pi \in \partial \T$.
		Moreover, when $f$ is continuous $u_f^*$ reaches $f$ at the boundary.
	\end{lem}

	\begin{proof}
		Note that $\mathcal{QC}_k(f)\neq \emptyset$. 
		Indeed, the constant function $u$ defined by 
		\[
			u(x) = \inf \{f(y) \colon y\in [0,1]\},
		\]
		is $k-$quasiconvex and bounded by $f$ on $\partial \T$, hence $\mathcal{QC}_k(f)\neq \emptyset$.
		
		To show that $u_f^*$ is also $k-$quasiconvex note that for any $u\in \mathcal{QC}_k(f)$ we have
		\begin{equation}
			\label{u_f^*-quasiconvex}
			u(x) \leq 
			\min_{\substack{y_1,\dots,y_k \in \S(x) \\ y_i \neq y_l}} 
				\, \max_{i=1,\dots,k} \left\{u(y_i)\right\}
						\leq
						\min_{\substack{y_1,\dots,y_k \in \S(x) \\ y_i \neq y_l}}
						\,  \max_{i=1,\dots,k} 
						\left\{u^*_f(y_i)\right\}
		\end{equation}
		at every $x\in\T$.
		Taking the supremum in~\eqref{u_f^*-quasiconvex}, it follows that 
		\[
		u^*_f (x) \leq
		\min_{\substack{y_1,\dots,y_k \in \S(x) \\ y_i \neq y_l}} \max_{i=1,\dots,k} 
						\left\{u^*_f(y_i)\right\}
		\]
		and we obtain that $u_f^*$ is $k-$quasiconvex using Theorem~\ref{teo.charac.local}.
		Moreover, one can use the comparison principle Theorem~\ref{cp-general} to show that every 
		$u \in \mathcal{QC}_k(f)$ verifies
		\[
		u(x) \leq  \sup \{f(y) \colon y\in [0,1]\}
		\]
		and then $u_f^*$ is well defined.

		Now, we aim to show that $u_f^*$ is bounded by $f$ on the boundary, that is, we want to show 
		\eqref{pepe}.
		First, note that for any $\pi\in\partial \T$ we have
		\begin{equation}
			\label{u_f^*-menor-f-borde}
				u_f^*(\pi) = \limsup_{x\to \pi} u_f^*(x) \leq f(\psi(\pi))
		\end{equation}
		due to the definition of $u_f^*$ as the supremum for functions in $\mathcal{QC}_k(f)$.
		This shows that $u_f^*$ in fact belongs to $\mathcal{QC}_k(f)$ and 
		therefore uniqueness of the $k-$quasiconvex envelope follows.
		This completes the proof for a bounded boundary datum $f$.
		
		We continue with the case where $f$ is continuous to show that 
		$u_f^*$ attains $f$ on $\partial\T$. Assume that for some $\pi\in\partial\T$ 
		the inequality \eqref{u_f^*-menor-f-borde} is strict and let $\varepsilon>0$ such that
		\begin{equation}
			\limsup_{x\to \pi} u_f^*(x)< f(\psi(\pi)) - \varepsilon.
		\end{equation}

		Let $j$ be such that on the interval $I_{\pi,j}$ one has
		\begin{equation}
			\min_{y\in I_{\pi,j}} f(y) > f(\psi(\pi)) - \frac{\varepsilon}{2},
		\end{equation}
		the existence of such an interval $I_{\pi,j}$ follows because $f$ is continuous.

		Now, recall that $I_{\pi,j}$ is a decreasing sequence of intervals such that 
		$\psi(\pi)\in I_{\pi,j}$ for every $j$. Denote by $x_j \in \T$ to the vertex for which 
		$I_{x_j} = I_{\pi,j}$ and write $\mathbb{T}_m^{\,x_j}$ for the subtree with regular 
		$m-$branching that has $x_j$ as root, that is, the subtree of $\T$ 
		containing all successors of $x_j$ of any level.

		Define $v \colon \T \to \R$ by the formula
		\begin{equation} \label{def.v}
			v(x) \coloneqq 	
				\begin{cases}
					\min\limits_{y\in I_{\pi,j}} f(y) & \text{ if } x\in \mathbb{T}_m^{\,x_j}, \\
					\min\limits_{y\in [0,1]} f(y) & \text{ otherwise. }
				\end{cases}
		\end{equation}
		We check that $v\in \mathcal{QC}_k(f)$. 
		It is immediate from \eqref{def.v} that $v$ is bounded above by $f$ on $\partial \T$.
		To show that $v$ is $k-$quasiconvex we just require to check the definition at $x_{j}$ and its 
		predecessor, $\hat{x}_j$.
		Since $\hat{x}_j\notin \mathbb{T}_m^{\,x_j}$ we have
		\[
			v(\hat{x}_j) 
			=
				\min\limits_{y\in [0,1]} f(y)
				\leq 
				\min_{\substack{y_1,\dots, y_k \in \S(\hat{x}_j) \\ y_i \neq y_l}} 
				\max_{i=1,\dots,k} \left\{v(y_i)\right\}
		\]
		because $v$ only at the successor $x_j \in \S(\hat{x}_j)$ takes a possible different value 
		from $v(\hat{x}_j)$ and this value is bigger or equal than $v(\hat{x}_j)$.
		We also have that 
		\[
			v(x_j) =
			\min\limits_{y\in I_{\pi,j}} f(y)
			= 
			\min_{\substack{y_1,\dots,y_k \in \S(x_j) \\ y_i \neq y_l}} 
			\max_{i=1,\dots,k} \left\{v(y_i)\right\} = \min\limits_{y\in I_{\pi,j}} f(y).
		\]
		A similar argument can be applied at any other vertex in $\T$.
		Hence, we conclude that $v\in \mathcal{QC}_k(f)$. 

		Then, $u_{f}^*(x) \geq v(x)$ for any $x\in\T$, so, in particular 
		\[
			u_f^*(x_j) 
			\geq 
			\min_{I_{\pi,j}} f(y) 
			> 
			f(\psi(\pi)) - \frac{\varepsilon}{2}.
		\]
		Hence, it follows that 
		\[
		\liminf_{x\to \pi} u_f^*(x)
		\geq
		f(\psi(\pi)) - \frac{\varepsilon}{2} 
		.
		\]
		Finally, we have obtained
		\[
		f(\psi(\pi)) - \frac{\varepsilon}{2} 
		\leq \liminf_{x\to \pi} u_f^*(x) \leq
		\limsup_{x\to \pi} u_f^*(x)
		<
		f(\psi(\pi)) - \varepsilon ,
		\]
		which contradicts our assumption over $\pi$ and completes the proof.
	\end{proof}

	The previous result showed that the quasiconvex envelope, $u^*_f$, is well defined. 
	Next, we look for the equation that it satisfies.  

	\begin{lem}\label{Lemma:eq-tree-quasiconvex}
		The $k-$quasiconvex envelope $u^*_f$ is the largest solution of the problem
		\begin{equation} 
		\label{eq-tree-quasiconvex-lema}
		\begin{cases}
			\displaystyle u (x)  = 
			\min_{\substack{y_1,\dots,y_k \in \S({x}) \\ y_i \neq y_j}} 
			\,
			\max_{i=1,\dots,k} 
			\left\{u(y_i)	\right\}  &\text{ for } x\in\T, 
			\\
			\displaystyle u(\pi) \leq f(\psi(\pi)) & \text{ for } \pi\in \partial \T.
			\end{cases}
		\end{equation}
	\end{lem}
	\begin{proof}
		Since $u^*_f$ is $k-$quasiconvex, we have that
		\[
			u^*_f(x)  \leq 
				\min_{\substack{y_1,\dots,y_k \in \S({x}) \\ y_i \neq y_j}} 
				\,\max_{i=1,\dots,k} \left\{u^*_f (y_i)\right\} ,
		\]
		for any $x\in\T$. Let us show that in fact, we have equality. 
		Arguing by contradiction, suppose that there exists $x\in\T$ for which 
		\[
			u_f^*(x) < \min_{\substack{y_1,\dots,y_k \in \S({x}) \\ y_i \neq y_j}} 
			\,\max_{i=1,\dots,k} \left\{u^*_f(y_i)\right\} ,
		\]
		and choose $\delta>0$ small enough such that adding $\delta$ to the left side 
		the inequality remains strict.
		Consider $v\colon \T \to \R$ defined by
		\[
			v(y) \coloneqq 
			\begin{cases}
 				u_f^*(y)& \text{ if } y\neq x, \\
  				u_f^*(x) + \delta & \text{ if } y=x.
			\end{cases}
		\]
		We claim that $v\in\mathcal{QC}_k(f)$. 
		To show this we just need to prove that $v$ is $k-$quasiconvex. 
		At $x,$ $v$ is $k-$quasiconvex by the choice of $\delta$. For $y \in \T\setminus\{x\}$ 
		we have
		\[
			v(y) = u_f^*(y) \leq \min_{\substack{y_1,\dots,y_k \in \S({x}) \\ y_i \neq y_j}} 
			\max_{i=1,\dots,k} \left\{u^*_f(y_i)\right\}
			\leq 
			\min_{\substack{y_1,\dots,y_k \in \S({x}) \\ y_i \neq y_j}} \max_{i=1,\dots,k} 
			\left\{v(y_i)\right\},
		\]
		since $v \geq u^*_f$.
		This proves $v\in\mathcal{QC}_k(f)$, but it contradicts the definition of $u_f^*$ as 
		the supremum of $\mathcal{QC}_k(f)$ since $v(x) > u_f^*(x)$.
		This proves that $u_f^*$ solves~\eqref{eq-tree-quasiconvex-lema}.

		Now, observe that any other function $u$ that solves~\eqref{eq-tree-quasiconvex-lema} 
		is $k-$quasiconvex and its below $f$ on $\partial \T$. 
		Hence, it follows that $u\in\mathcal{QC}_k(f)$ and we must have $u\leq u_f^*$. 
		Therefore, $u_f^*$ is the largest solution of~\eqref{eq-tree-quasiconvex-lema}.
	\end{proof}

	Now, we are ready to end the proof of Theorem \ref{teo-quasiconvex}.

	\begin{proof}[Proof of Theorem \ref{teo-quasiconvex}] 
		We have already proved that the $k-$quasiconvex envelope $u^*_f$ is the largest solution 
		to \eqref{eq-tree-quasiconvex-lema}. 
		
		To finish the proof we need to show that it is the unique solution to
		\eqref{eq-tree-quasiconvex-lema} when the boundary datum $f$ is continuous.
		Assume that $v$ is a solution to~\eqref{eq-tree-quasiconvex-lema} that reaches 
		the boundary condition. So, $v$ is $k-$quasiconvex and from our previous result 
		we have that
		\[
			v(x) \leq u_f^*(x) \qquad \text{for all }x\in\T.
		\]
		Since $v$ and $u_f^*$ coincide on $\partial \T$, 
		we show the opposite inequality applying the comparison principle Theorem~\ref{cp-general}.
		Then, $v=u_f^*$ and the problem~\eqref{eq-tree-quasiconvex-lema} has a unique solution 
		for continuous boundary conditions.
	\end{proof}

\subsection{The quasiconvex envelope of a function inside $\T$}
	Finally, we include the proofs for the $k-$quasiconvex envelope of a function $g\colon \T \to \R$. 
	Recall that the $k-$quasiconvex envelope of $g$, $u_g^{\star} $, is given by
	\begin{equation}
	\label{solucion-obstaculo.99}
	u_g^{\star} (x) = \sup \left\{u(x) \colon u\in \mathfrak{QC}_k(g)\right\},
	\end{equation}
	where
	\begin{equation}
		\mathfrak{QC}_k(g) \coloneqq \left\{u\colon \T\to \R \colon u 
		\text{ is $k-$quasiconvex and } u(x)\leq g(x) \quad \forall x\in\T \right\}.
	\end{equation}
	When we assume that $g$ is bounded below the $k-$quasiconvex envelope is well defined
	since $u \equiv \inf g \in \mathfrak{QC}_k(g)$.
	Recall that the $k-$quasiconvex envelope $u_g^{\star} $ is also unique 
	(this fact can be proved exactly as we did for the quasiconvex envelope of a boundary datum).

	One can characterize $u_g^{\star} $ as the solution to the obstacle problem for 
	the equation \eqref{def-quasiconvex-equiv}. This is the content of Theorem \ref{teo-obstaculo} 
	that we prove next.

	\begin{proof}[Proof of Theorem~\ref{teo-obstaculo}]
		First, note that there is at least one $k-$quasiconvex function bounded above by $g$ on $\T$.
		In fact, the function $u \colon \T\to\R$ given by $u(x) = \inf \{g(y) \colon y\in\T\}$ is well-defined 
		because $g$ is bounded below, and $u$ is also $k-$quasiconvex.
		Then, $\mathfrak{QC}_(g) \neq \emptyset$. 
		In addition, we have that every $v\in \mathfrak{QC}_k(g) $ verifies
		\[
			v(x) \leq g(x) \leq \sup \{g(y) \colon y\in\T\} < +\infty.
		\]		
		Hence, $u_g^{\star} $ is well-defined and bounded above by $g$.

		To show that $u_g^{\star}$ is $k-$quasiconvex, we use that for any $u\in \mathfrak{QC}(g)$
		\[
			u(x) \leq 
			\min_{\substack{y_1,\dots,y_k \in \S(x) \\ y_i \neq y_j}} 
			\, \max_{i=1,\dots,k} \left\{u (y_i)\right\}  \leq 
			\min_{\substack{y_1,\dots,y_k \in \S(x) \\ y_i \neq y_j}} \max_{i=1,\dots,k}
			\left\{u_g^{\star} (y_i)\right\}.
		\]
		So, taking supremum over $u\in \mathfrak{QC}_k(g)$ the $k-$quasiconvexity 
		of $u_g^{\star} $ follows. This proves that $u_g^{\star}$ is a solution of the obstacle problem.

		Let $v^\star$ be the largest solution of the obstacle problem
		\begin{equation}
			\label{obstaculo-ecuacion-lema}
				\begin{cases}
					\displaystyle u (x)  \leq  
					\min_{\substack{y_1,\dots,y_k \in \S(x) \\ y_i \neq y_j}}\, 
					\max_{i=1,\dots,k} \left\{v (y_i)\right\} &\text{ for } x\in\T\\
			 		\displaystyle  v(x) \leq g(x) & \text{ for } x \in \T,
			\end{cases}
		\end{equation}
		that is,
		\begin{equation}\label{ec.v^*}
			v^\star(x) = \sup \left\{v(x) \colon v \mbox{ satisfies  \eqref{obstaculo-ecuacion-lema}} \right\}.
		\end{equation}
		Since $u_g^{\star} $ satisfies \eqref{obstaculo-ecuacion-lema} we have $u_g^{\star} \leq v^\star(x)$. 
		Our goal now is to prove that 
		\[
			v^*(x) = u_g^{\star} (x)
		\] 
		for every $x\in\T$.
		Note that by definition $v^\star$ is $k-$quasiconvex and bounded above by $g$ at all vertices. 
		So, $v^\star\in \mathfrak{QC}_k(g)$ and therefore $v^\star(x)\leq u_g^{\star} (x)$ for any 
		$x\in\T$.

		It remains to prove the claims about the coincidence set $CS(g)$. 
		If $x\in CS(g)$, using that $u_g^{\star} \leq g$ at any vertex, we have  
		\[
			g(x) =
			u_g^{\star} (x) 
			\leq 
			\min_{\substack{y_1,\dots,y_k \in \S(x) \\ y_i \neq y_j}} \max_{i=1,\dots,k}  
			\left\{u_g^{\star} (y_i)\right\}
			\leq 
			\min_{\substack{y_1,\dots,y_k \in \S(x) \\ y_i \neq y_j}} \max_{i=1,\dots,k}  
			\left\{g (y_i)\right\},
			\]
		and it follows that $g$ verifies the inequality~\eqref{coincidence-set} in the coincidence set $CS(g)$.

		For the complement of $CS(g)$ we want to prove that $u_g^{\star}$ 
		satisfies the equation \eqref{fuera-coincidence-set}. 
		Arguing by contradiction, suppose that for some $x\not\in CS(g)$ we have
		\[
		u_g^{\star} (x)
		< 
		\min_{\substack{y_1,\dots,y_k \in \S(x) \\ y_i \neq y_j}} \max_{i=1,\dots,k}
		\left\{u_g^{\star} (y_i)\right\}.
		\]
		Thus, adding a small $\delta>0$ to the left-hand side the previous inequality remains strict.
		Therefore, the function
		\[
		v(y) = 
			\begin{cases}
			u_g^{\star} (y) & \text{ for } y\neq x, \\
			u_g^{\star} (x) + \delta & \text{ for } y=x.
			\end{cases}
		\]
		is $k-$quasiconvex (see the proof of Lemma \ref{Lemma:eq-tree-quasiconvex}) 
		and we still have $v\leq g$, contradicting the maximality assumption of $u_g^{\star} $.
	\end{proof}

	The proof of Corollary \ref{corol.compar.interior} is analogous to the one of Corollary \ref{corol.compar}
	and thus it is left to the reader.

\section{The quasiconvex envelope for monotone boundary data} \label{sect-examples}

	In this section, we present a simple example that illustrates that quasiconvex envelopes are easy to 
	compute when the boundary data are monotone (or piecewise monotone).

	\begin{ex} Assume that $f\colon[0,1] \mapsto \mathbb{R}$ is continuous and monotone increasing 
		(the case when $f$ is decreasing is completely analogous).
		Then, the $k-$quasiconvex envelope of the boundary datum $f$ inside $\T$ is given by
		\[
			u^*_f(x) = f(\psi(x,k-1,k-1,\dots)).
		\]
		In fact, one can check that $u^*_f$ is a 
		solution to the problem that characterizes the $k-$quasiconvex 
		envelope, \eqref{eq-tree-quasiconvex-lema}, that is, $u$ satisfies
		\begin{equation} 
			\label{eq-tree-quasiconvex-lema.77}
			\begin{cases}
				\displaystyle u (x)  = 
				\min_{\substack{y_1,\dots,y_k \in \S(x) \\ y_i \neq y_j}} \max_{i=1,\dots,k}			
				\left\{u(y_i)\right\}  &\text{ for } x\in\T, \\
				\displaystyle u(\pi) = f(\psi(\pi)) & \text{ for } \pi\in \partial \T.
			\end{cases}
		\end{equation}
		In fact, since we have that $f$ is increasing, it holds that
		\begin{align*}
			u(x,0) = f(\psi(x,0,k-1\dots)) 
			\leq u(x,1) &= f(\psi(x,1,k-1\dots)) \leq\dots \\ & \,\dots \leq u(x,m-1)= f(\psi(x,m-1,k-1\dots)),
		\end{align*}
		and hence is the $k-$th smallest value of $u$ on $\S(x)$ is attained at 
		$y=(x,k-1)$ and the equation \eqref{eq-tree-quasiconvex-lema.77} is satisfied.

		Moreover, since $f$ is continuous we get that
		\[
			\lim_{x\to \pi} u(x) = \lim_{x\to \pi} f(\psi(x,k-1,k-1,\dots)) = f(\psi (\pi)).
		\]

		When $f$ has a finite number of maximums/minimums 
		($f$ is increasing in some subintervals and decreasing in others), 
		then the previous idea can be applied to obtain the values of $u$ at the nodes $x$ such that 
		$\psi(x)$ is at the interior of such intervals. 
		Details are left to the reader.
		
		Therefore, one can construct a solution to the problem for the quasiconvex envelope 
		\eqref{eq-tree-quasiconvex-lema.77} by approximation. 
		Fix a continuous boundary datum $f$ and approximate it by the piecewise linear continuous 
		function $f_n$ that interpolates $f$ at the points $\psi(x_i) = (i-1)/m^n$, $i=1,\dots,m^n+1$. 
		These approximating functions $f_n$ are monotone on every interval $[x_i,x_{i+1}]$ 
		and converge uniformly to $f$ as $n\to \infty$.
		Therefore, one can compute its quasiconvex envelope 
		$u^*_{f_n}$ to obtain the values inside the intervals $[x_i,x_{i+1}]$.
		In the other nodes of $\T$, that are a finite number, one computes the values solving backward the 
		finite system of equations, and then pass to the limit as $n \to \infty$. 
		A similar approximation argument can be found in \cite{DPMR1,DPMR2}.
 
		This approximation of a solution to \eqref{eq-tree-quasiconvex-lema.77} is well suited for explicit 
		computations, and also shows existence (and uniqueness follows from the comparison argument) of a 
		solution. 
		However, it is not immediate from this construction that solutions 
		to \eqref{eq-tree-quasiconvex-lema.77} are related to the quasiconvex envelope of $f$
		inside $\T$. 
		For this reason, we prefer to construct the solution (or the quasiconvex envelope) as a supremum of 		functions that verify an inequality and are below $f$ on $\partial \T$.
	\end{ex}
	
	Finally, we remark that for a function $g:\T \mapsto \mathbb{R}$ such that $g$ is increasing in the sense that
	$$
	g(x) \leq g(y) \qquad \mbox{when } \psi(x) \leq \psi(y), 
	$$
	we have that 
	$$g(x) \mbox{ is $k-$quasiconvex (and also $k-$convex)} $$
	for any $k$. 
	This holds since
			$$
			g(x) \leq g (y) \qquad \forall y \in \S(x).
			$$
			Therefore, the $k-$quasiconvex envelope of $g$, $u^\star_g(x)$,
			coincides with $g$ in $\T$ and $CS(g)=\T$.
			
			This has to be contrasted with the fact that for an increasing function $f$ as boundary datum
			on $\partial \T$ we have $u^*_f(x)\le f(\psi(x))$ in $\T$ and the inequality is strict when $f$ is
			strictly increasing.
			
			As we mentioned in the introduction, the cases $k=1$ and $k=m$ are simpler (and hence less interesting).
			In the special case $k=1$ the equation for the quasiconvex envelope is
			$$
			u (x)  =  \min_{y\in \S(x) } \left\{u(y)\right\}.
			$$
			In this special case, the quasiconvex envelope of a boundary datum $f$ is given by
			$$
			u^*_f (x) = \inf_{z\in I_x} f(z).
			$$
			Analogously, for $k=m$ we have that the associated equation is
			$$
			u (x)  =  \max_{y\in \S(x) } \left\{u(y)\right\},
			$$
			and in this case, the quasiconvex envelope of a boundary datum $f$ is 
			$$
			u^*_f (x) = \sup_{z\in I_x} f(z).
			$$

		\section{On the equations that characterize the convex and quasiconvex envelopes} \label{sect-PDEs}

	In this section, we aim to compare and obtain a complete analogy between the equations 
	that appear when one considers the convex and quasiconvex envelopes on $\mathbb{R}^N$ and $\T$. 
	For the tree case, to obtain a complete analogy, we consider $k=2$, that is, we look for the $2-$convex and the $2-$quasiconvex
	envelopes of a boundary datum.

	\subsection{The Euclidean case} 
		In $\mathbb{R}^N$ let us first recall that the usual Laplacian is given by
		\[
			\Delta u(x) = \sum_{i=1}^N \frac{\partial^2 u}{\partial x_i^2}(x) 
			= \sum_{i=1}^N \lambda_i(D^2u(x)).
		\]
		Here, and in what follows, $\lambda_1\leq \lambda_2\leq\cdots\leq\lambda_N$ 
		are the ordered eigenvalues of the Hessian matrix, $D^2u$.
		That is, the Laplacian is given by the sum of the pure second derivatives 
		or by the sum of the eigenvalues of the Hessian matrix. 

		We have that the convex envelope inside a domain $\Omega\subset
		\mathbb{R}^N$ turns out to be a solution to
		\begin{equation} \label{convex-envelope-usual-eq}
			\begin{array}{ll}
			\lambda_1 (D^2 u) (x) = 0 \qquad &x\in  \Omega, 
			\end{array}
		\end{equation}
	where the equation has to be interpreted in viscosity sense. 
	Here $\lambda_1(D^2u)$ is the smallest of the eigenvalues of $D^2u$. 
	We refer to \cite{BlancRossi,OS,Ober33}.
	Notice that the equation \eqref{convex-envelope-usual-eq} is equivalent to 
	\begin{equation} \label{nacional}
		\min_{|v|=1} \langle D^2 u(x) v, v \rangle =0.
	\end{equation}
	This says that the equation that governs the convex envelope is just the minimum among all possible 
	directions of the second derivative of the function at $x$ equal to zero. 
	
	As we mentioned in the introduction, there is a PDE for the quasiconvex envelope, see 
	\cite{BGJ12a,BGJ12b,BGJ13}. 
	In fact, the quasiconvex envelope of a boundary datum $u$ in the Euclidean space is a viscosity solution to 
	\begin{equation}\label{ec-RN.55}
			\min_{\substack{|v|=1 \\ \langle v, \nabla u(x) \rangle =0}} 
			\langle D^2 u(x) v, v \rangle =0.
	\end{equation}
	In words, the equation for the quasiconvex envelope involve 
	the minimum of the second derivatives in directions that are perpendicular to the gradient of the solution. 

	Finally, let us say that the infinity Laplacian in the Euclidean setting is given by the nonlinear operator
	\[
		\Delta_\infty u(x) \coloneqq \langle D^2u(x) \nabla u(x), \nabla u(x) \rangle,
	\]
	that is, the infinity Laplacian is given by the second derivative in the direction of the gradient of the 
	function.

\subsection{The tree case} 
	Fix $k=2$. The usual Laplacian on $\T$ is defined by the mean value formula, 
	\begin{equation} \label{hhh}
		\Delta u (x) \coloneqq 
		\frac1m \sum_{y\in\S(x)} \big(u(y) - u(x) \big),
		\qquad\forall x\in\T,
	\end{equation}
	see for instance \cite{KLW}.
	In \cite{DPFRconvex} we introduce a notion of convexity based on the same idea of 
	``{}segments"{} that we used here, using finite binary trees as segments. 
	The convex envelope of a boundary datum on $\T$ with this setting satisfies the equation
	\[
		0 = \min_{\substack{y_i,y_j\in\mathcal{S}(x)\\ y_i\neq y_j}} 
        \left\{ \frac12 
        u(y_i) + \frac12 u(y_j) - u(x) \right\}.
    \]
	In this case, in clear analogy with \eqref{nacional}, we can identify the analogous to the eigenvalues of 
	the Hessian that are given by,
    \begin{equation} \label{uu.II}
        \left\{ 
            \frac12 u(y_i) + \frac12 u(y_j) - u(x) 
        \right\}_{i \neq j} .
    \end{equation}
	Then, taking the average we obtain the usual Laplacian given by \eqref{hhh}, where we look at the Laplacian 
	as the sum of the eigenvalues of the Hessian.

	A different way to find the Laplacian runs as follows: 
	fix $x\in \T$ and think about it as the midpoint between two successors $y_i, y_j\in\S(x)$ on the tree.
	So, computing the finite central difference approximation	
	\[
		\frac12 u(y_i) + \frac12 u(y_j) - u(x),
	\]
	we can understand it as a ``{}mixed second derivative"{} in the directions from $x$ to 
	$y_i$ and from $x$ to $y_j$.
	Then, the pure second derivative in the direction of $y\in \S(x)$ is given by
	\[
		u(y)-u(x).
	\]
	Adding these pure second derivatives in every direction, that is, for every successor $y$, and dividing by 
	$m$, we obtain again the usual Laplacian given by \eqref{hhh} but now interpreted as the sum of the pure 
	second derivatives.

	One the other hand, following \cite{Ober,s-tree1} we have that the infinity Laplacian in the tree is given 
	by
	\[
		u(x) = \frac12 \max_{y \in \S({x})} u(y) + \frac12 \min_{y \in \S({x})} u(y).
	\]
	Therefore, we can identify the ``{}direction of the gradient"{} of a function defined in the tree as the 
	two directions given by the successors at which $\max_{y \in \S({x})} u(y)$ and $\min_{y \in \S({x})} u(y)$ 
	are attained.
	
	Now, for the $2-$convex envelope of a boundary datum on $\T$, in \cite{DPFRconvex} it is shown that the associated equation 
is given by 
\[
    \min_{\substack{y_1,y_2 \in \S({x})\\ y_1 \neq y_2 }}  \frac12 u(y_1) + \frac12 u(y_2) - u(x)=0, 
\]
    that is, the minimum of the second eigenvalues of the Hessian (see \eqref{uu.II} and c.f. \eqref{convex-envelope-usual-eq}).

	Finally, for the $2-$quasiconvex envelope of a boundary datum on $\T$, we have proved that it satisfies the 
	equation
	\[
		u (x)  = \min_{\substack{y_1,y_2 \in \S({x}) \\ y_1 \neq y_2}} 
		\max_{i=1,2} \left\{u(y_i)\right\},
	\]
	that is, the value $u(x)$ is the second smallest value of $u$ on $\S(x)$.

	If we consider pure second derivatives of $u$ in directions ``{}orthogonal to the direction 
	of the gradient"{}, in an analogy with the equation for the $2-$quasiconvex envelope in the Euclidean 
	setting \eqref{ec-RN.55}, and we compute the minimum, we obtain
	\begin{equation}
		\min_{\substack{y_i \in \S({x}) \\ y_i \neq y^*, \, y_i \neq y_*}}  
		\left\{u(y_i) -u(x) \right\},
	\end{equation}
	where $y^*,y_*$ are the two successors at which the $\max_{y \in \S({x})} u(y)$ and the 
	$\min_{y \in \S({x})} u(y)$ are attained.
	Note that the last expression can be rewritten as
	\[
		\min_{\substack{y_i \in \S(x) \\ y_i \neq y^*, \, y_i \neq y_*}}  
		\left\{u(y_i) -u(x) \right\} = 
		\min_{\substack{y_1,y_2 \in \S(x) \\ y_1 \neq y_2}} \max_{i=1,2} 
		\left\{u (y_i) -u(x)\right\}
		= \min_{\substack{y_1,y_2 \in \S(x) \\ y_1 \neq y_2}} \max_{i=1,2} \left\{u (y_i)\right\}  -u(x),
	\]
	and we have interpreted our equation for the $2-$quasiconvex envelope on $\T$ as the minimum of the second 
	derivatives of $u$ in directions that are orthogonal to the direction of the gradient of $u$.

        \medskip

{\bf Acknowledgements.}  \

Supported by CONICET grant PIP GI No 11220150100036CO (Argentina), by  UBACyT grant 20020160100155BA (Argentina), by MINECO MTM2015-70227-P (Spain), by MATH-AmSud 22-MATH-04, and by FVF-2021-063–DICYT (Uruguay).

\end{document}